\documentclass[11pt,reqno]{amsart}
\usepackage{fullpage}
\usepackage{mathrsfs,amssymb,graphicx,verbatim,amsmath,amsfonts}
\usepackage{paralist}
\usepackage[breaklinks,pdfstartview=FitH]{hyperref}
\usepackage{upgreek}
\usepackage{tikz}
\usepackage[mathscr]{euscript}
\usepackage{relsize}
 \usepackage{amsaddr}

\makeatletter
\@namedef{subjclassname@2020}{%
  \textup{2020} Mathematics Subject Classification}
\makeatother

\usepackage{dutchcal}

\newcommand{\MM}{\mathcal{M}}

\DeclareSymbolFont{rmlargesymbols}{OMX}{mdbch}{m}{n}
\DeclareMathSymbol{\rmintop}{\mathop}{rmlargesymbols}{82}
\newcommand{\rmint}{\rmintop\nolimits}

\hypersetup{ocgcolorlinks=true,allcolors=testc}
\hypersetup{
     colorlinks   = true,
     citecolor    = black
}
\hypersetup{linkcolor=black}
\hypersetup{urlcolor=black}

\renewcommand{\leq}{\leqslant}
\renewcommand{\geq}{\geqslant}
\renewcommand{\setminus}{\smallsetminus}
\renewcommand{\gamma}{\upgamma}
\allowdisplaybreaks

\usepackage{esint}
\usepackage[autostyle]{csquotes}

\renewcommand{\pi}{\uppi}
\newcommand{\NN}{\mathcal{N}}

\newcommand{\sign}{\mathrm{sign}}

\newcommand{\e}{\varepsilon}
\newcommand{\R}{\mathbb R}

\newtheorem{theorem}{Theorem}
\newtheorem{lemma}[theorem]{Lemma}
\newtheorem{proposition}[theorem]{Proposition}

\newtheorem{corollary}[theorem]{Corollary}

\theoremstyle{remark}
\newtheorem{remark}[theorem]{Remark}

\renewcommand{\tau}{\uptau}

\renewcommand{\xi}{\upxi}
\renewcommand{\rho}{\uprho}

\renewcommand{\subset}{\subseteq}
\newcommand{\C}{\mathbb C}

\renewcommand{\hat}{\widehat}

\newcommand{\N}{\mathbb N}

\newcommand{\eqdef}{\stackrel{\mathrm{def}}{=}}

\renewcommand{\theta}{\uptheta}
\renewcommand{\lambda}{\uplambda}

\renewcommand{\gamma}{\upgamma}
\renewcommand{\beta}{\upbeta}
\renewcommand{\alpha}{\upalpha}
\renewcommand{\kappa}{\upkappa}
\renewcommand{\psi}{\uppsi}
\renewcommand{\rho}{\varrho}
\renewcommand{\delta}{\updelta}
\renewcommand{\pi}{\uppi}
\renewcommand{\omega}{\upomega}

\renewcommand{\eta}{\upeta}

\renewcommand{\kappa}{\upkappa}
\renewcommand{\mu}{\upmu}
\renewcommand{\nu}{\upnu}
\renewcommand{\pi}{\uppi}
\renewcommand{\zeta}{\upzeta}

\newcommand{\mb}{\mathbb}
\newcommand*\diff{\mathop{}\!\mathrm{d}}

\newcommand{\ms}{\mathscr}
\newcommand{\msf}{\mathsf}
\newcommand{\mr}{\mathrm}

\begin{document}

\title{Discrete logarithmic Sobolev inequalities in Banach spaces}
\author{Dario Cordero-Erausquin and Alexandros Eskenazis}
\address{Institut de Math\'ematiques de Jussieu, Sorbonne Universit\'e, Paris, 75252, France}
\thanks{{\it E-mail addresses:} \href{mailto:dario.cordero@imj-prg.fr}{\nolinkurl{dario.cordero@imj-prg.fr}}, $\{$\href{alexandros.eskenazis@imj-prg.fr}{\nolinkurl{alexandros.eskenazis@imj-prg.fr}}, \href{ae466@cam.ac.uk}{\nolinkurl{ae466@cam.ac.uk}}$\}$.
}


\vspace{-0.25in}

\begin{abstract} 
Let $\mathscr{C}_n=\{-1,1\}^n$ be the discrete hypercube equipped with the uniform probability measure $\sigma_n$. We prove that if $(E,\|\cdot\|_E)$ is a Banach space of finite cotype and $p\in[1,\infty)$, then every function $f:\mathscr{C}_n\to E$ satisfies the dimension-free vector-valued $L_p$ logarithmic Sobolev inequality
$$\big\|f-\mathbb{E} f\big\|_{L_p(\log L)^{p/2}(E)} \leq\mathsf{K}_p(E) \left( \rmint_{\mathscr{C}_n} \Big\| \sum_{i=1}^n \delta_i \partial_i f\Big\|_{L_p(E)}^p \,\diff\sigma_n(\delta)\right)^{1/p}.$$
The finite cotype assumption is necessary for the conclusion to hold.  This estimate is the hypercube counterpart of a result of Ledoux (1988) in Gauss space and the optimal vector-valued version of a deep inequality of Talagrand (1994). As an application, we use such vector-valued $L_p$ logarithmic Sobolev inequalities to derive new lower bounds for the bi-Lipschitz distortion of nonlinear quotients of the Hamming cube into Banach spaces with prescribed Rademacher type.
\end{abstract}
\maketitle


\section{Introduction}

Let $\ms{C}_n$ be the Hamming cube $\{-1,1\}^n$ equipped with the uniform probability measure $\sigma_n$ and let $\psi:[0,\infty)\to[0,\infty)$ be a strictly increasing, convex function with $\psi(0)=0$. If $(E,\|\cdot\|_E)$ is a normed space, we shall denote the $\psi$-Orlicz norm norm of an $E$-valued function $f:\ms{C}_n\to E$ by
\begin{equation}
\|f\|_{L_\psi(E)} \eqdef \inf\left\{ \gamma>0: \  \rmint_{\ms{C}_n} \psi\big( \|f(x)\|_E/\gamma\big) \,\diff\sigma_n(x) \leq 1\right\}.
\end{equation}
Clearly,  if $\psi(t)=t^p$,  then the $L_\psi(E)$ norm reduces to the vector-valued $L_p(E)$ norm.  Per standard notation, if $\psi(t) = t^p \log^\alpha(e+t^p)$ for some $\alpha>0$,  we denote $L_\psi(E)$ by $L_p(\log L)^\alpha(E)$.  
Moreover,  we will denote by $\mb{E} f$ the expectation of $f$ with respect to $\sigma_n$ and when $E=\C$ we shall abbreviate $\|f\|_{L_p(\log L)^\alpha(\C)}$ as $\|f\|_{L_p(\log L)^\alpha}$.  If $i\in\{1,\ldots,n\}$, the $i$-th partial derivative of $f$ is defined as
\begin{equation}
\forall \ x\in\ms{C}_n, \qquad \partial_if (x) \eqdef \frac{f(x) - f(x_1,\ldots,x_{i-1},-x_i,x_{i+1},\ldots,x_n)}{2}.
\end{equation}
When $E=\C$,  we will denote the $L_p$ norm of the gradient of a function $f:\ms{C}_n\to\C$ by
\begin{equation} \label{eq:grad}
\|\nabla f\|_{L_p} \eqdef \left( \rmint_{\ms{C}_n} \Big( \sum_{i=1}^n |\partial_if(x)|^2\Big)^{p/2} \,\diff\sigma_n(x)\right)^{1/p}.
\end{equation}

Motivated by a quantitative version of Margulis' graph connectivity theorem \cite{Mar74}, Talagrand \cite{Tal93} proved an important functional inequality asserting that for every $p\in[1,\infty)$ there exists $\msf{C}_p\in(0,\infty)$ such that for any $n\in\N$, every function $f:\ms{C}_n\to\C$ satisfies
\begin{equation} \label{eq:talagrand}
\|f- \mb{E}f\|_{L_p(\log L)^{p/2}} \leq \msf{C}_p \|\nabla f\|_{L_p}.
\end{equation}
When $p=2$, this estimate is equivalent to the usual logarithmic Sobolev inequality on the discrete hypercube up to the value of the implicit constant which follows from the classical works \cite{Bon70, Gro75} of Bonami and Gross (see also Section \ref{sec:entropy} for comparison with the more classical formulation of this inequality in terms of entropy).  It is worth pointing out that $L_p(\log L)^{p/2}$ is the optimal Orlicz space guaranteeing the validity of an inequality in the spirit of \eqref{eq:talagrand}. As a direct consequence of the $L_p$ logarithmic Sobolev inequality \eqref{eq:talagrand}, Talagrand also discovered the first proof of the dimension-free $L_p$ Poincar\'e inequality on the hypercube, according to which every $f:\ms{C}_n\to\C$ satisfies
\begin{equation} \label{eq:poincare}
\|f- \mb{E}f\|_{L_p} \leq \msf{C}_p \|\nabla f\|_{L_p}.
\end{equation}
Refinements and extensions of \eqref{eq:poincare} have been a central topic of research in discrete analysis after Talagrand's original proof (see, e.g., the works \cite{Bob97,BG99,NS02,HN13,INV18,Esk21,IVV20,EG22,CE22}).

In a striking recent article, Ivanisvili, van Handel and Volberg \cite{IVV20}, obtained a vector-valued extension of \eqref{eq:poincare} with optimal assumptions on the target space.  In \cite[Theorem~1.5]{IVV20}, they showed that if a Banach space $E$ has finite cotype (see \cite[Chapter~9]{LT91}) then for any $p\in[1,\infty)$ there exists a constant $\msf{C}_p(E)$ such that for any $n\in\N$,  every function $f:\ms{C}_n\to E$ satisfies
\begin{equation} \label{eq:ivv}
\|f- \mb{E}f\|_{L_p} \leq \msf{C}_p(E)  \left( \rmint_{\ms{C}_n} \Big\| \sum_{i=1}^n \delta_i \partial_i f\Big\|_{L_p(E)}^p \,\diff\sigma_n(\delta)\right)^{1/p}.
\end{equation}
Observe that this inequality indeed extends \eqref{eq:poincare} since, when $E=\C$, the right-hand side of \eqref{eq:ivv} is equivalent to $\|\nabla f\|_{L_p}$ up to constants depending only on $p$ by Khintchine's inequality \cite{Khi23}. Vector-valued inequalities of the form \eqref{eq:ivv} where first studied in classical work of Pisier \cite{Pis86},  who showed that \eqref{eq:ivv} is valid with $\msf{C}_p(E)$ replaced by a constant multiple of $\log n$ for an arbitrary target Banach space $E$. In \cite[Section~6]{Tal93}, Talagrand constructed a Banach space for which the logarithmic factor cannot be removed and thus the finite cotype assumption is necessary for the validity of the dimension-free inequality \eqref{eq:ivv} in view of the Maurey--Pisier theorem \cite{MP76}.

In contrast to the Poincar\'e inequality \eqref{eq:poincare} which has been extensively studied in the context of analysis in Banach spaces,  vector-valued versions of Talagrand's logarithmic Sobolev inequality \eqref{eq:talagrand} have been elusive (even for the most well-known special case $p=2$). The main purpose of the present article is to unify Talagrand's work \cite{Tal93} with the work of Ivanisvili, van Handel and Volberg \cite{IVV20} by proving an \emph{optimal} vector-valued version of \eqref{eq:talagrand}. Our main result reads as follows.

\begin{theorem} \label{thm:main}
Let $(E,\|\cdot\|_E)$ be a Banach space of finite cotype and $p\in[1,\infty)$.  There exists a constant $\msf{K}_p(E)\in(0,\infty)$ such that for any $n\in\N$, every function $f:\ms{C}_n\to E$ satisfies
\begin{equation} \label{eq:main}
\big\|f-\mb{E}_{}f\big\|_{L_p(\log L)^{p/2}(E)} \leq \msf{K}_{p}(E) \left( \rmint_{\ms{C}_n} \Big\| \sum_{i=1}^n \delta_i \partial_i f\Big\|_{L_p(E)}^p \,\diff\sigma_n(\delta)\right)^{1/p}.
\end{equation}
\end{theorem}

Our proof in fact yields a stonger dimension-free inequality in an arbitrary Banach space (see Remark \ref{rem:general}) in the spirit of \cite[Theorem~1.4]{IVV20}. This allows us to derive the following refinement of the classical Pisier inequality \cite[Lemma~7.3]{Pis86} \mbox{which may be of independent interest.}

\begin{corollary} \label{cor:pisier}
For every $p\in[1,\infty)$ there exists $\msf{K}_p\in(0,\infty)$ such that the following claim holds for an arbitrary Banach space $(E,\|\cdot\|_E)$.  For any $n\in\N$, every function $f:\ms{C}_n\to E$ satisfies 
\begin{equation} \label{eq:strong-pisier}
\begin{split}
\big\|f-\mb{E}_{}f\big\|_{L_p(\log L)^{p/2}(E)} \leq  \msf{K}_p & \left( \rmint_{\ms{C}_n} \Big\| \sum_{i=1}^n \delta_i \partial_i f\Big\|_{L_p(E)}^p \,\diff\sigma_n(\delta)\right)^{1/p} 
\\ & +2(\log n+1)\rmint_{\ms{C}_n} \Big\| \sum_{i=1}^n \delta_i \partial_i f\Big\|_{L_1(E)} \,\diff\sigma_n(\delta).
\end{split}
\end{equation}
\end{corollary}

It is worth emphasizing that \eqref{eq:strong-pisier} is a two-fold refinement of Pisier's inequality.  Besides the fact that  it allows us to derive an upper bound for the $L_p(\log L)^{p/2}$ Orlicz norm of the  function $f-\mb{E}f$ instead of the usual $L_p$ norm, the dimensional factor $\log n$ is multiplied by the vector-valued $L_1$ norm of the gradient of $f$ rather than the vector-valued $L_p$ norm as in \cite{Pis86}.


\subsection{Self improvement of Poincar\'e inequalities for hypercontractive semigroups} \label{sec:intro-boost}
As we shall explain in detail in Section \ref{sec:pre}, the main intricacies arising in the proof of Theorem \ref{thm:main} are due to the appearance of the vector-valued gradient on the right-hand side of inequality \eqref{eq:main}.  In contrast to that, much simpler reasoning allows us to prove that \emph{metric} Poincar\'e inequalities (that is, inequalities in which the right-hand side depends only on distances between values of the function) for vector-valued functions can be boosted to metric logarithmic Sobolev type inequalities provided that the additional scalar-valued inequality holds. This implication is summarized as
\begin{center}
{\it scalar log-Sobolev + vector Poincar\'e \qquad $\Longrightarrow$ \qquad vector log-Sobolev}
\end{center}
and will be formalized in Proposition \ref{prop:boost} below.  For now, we register one consequence of it which is the first vector-valued logarithmic Sobolev inequality on the symmetric group $\ms{S}_n$ generated by transpositions for targets of martingale type 2 (see Section \ref{sec:boost} for all the relevant definitions).

\begin{theorem} \label{thm:symmetric}
Let $\ms{S}_n$ be the symmetric group on $n$ elements equipped with the uniform probability measure $\nu_n$ and $\ms{T}_n\subset \ms{S}_n$ the generating set consisting of all transpositions. If $(E,\|\cdot\|_E)$ is a Banach space of martingale type 2, then there exists a constant $\msf{M}(E)\in(0,\infty)$ such that for any $n\in\N$, every function $f:\ms{S}_n\to E$ satisfies
\begin{equation} \label{eq:symmetric}
\big\| f-\mb{E}f\big\|_{L_2(\log L)(E)}^2 \leq \frac{\msf{M}(E) \log n}{n-1} \sum_{\tau\in\ms{T}_n} \rmint_{\ms{S}_n} \big\|f(\pi) - f(\tau \circ\pi) \big\|_E^2\,\diff\nu_n(\pi) .
\end{equation}
\end{theorem}

This corollary improves upon the corresponding Poincar\'e inequality in the same setting which follows from the work of Khot and Naor \cite[Section~6]{KN06} up to the (necessary) logarithmic factor.

Recall that if $(\Omega,\mu)$ is a measure space, the entropy of a nonnegative function $h:\Omega\to\R_+$ is
\begin{equation}
\mr{Ent}(h) \eqdef \rmint_\Omega h\log h\, \diff\mu - \left( \rmint_\Omega h\,\diff\mu\right) \log\left( \rmint_\Omega h\,\diff\mu\right).
\end{equation}
Logarithmic Sobolev inequalities like \eqref{eq:symmetric} are usually stated in terms of entropy rather than Orlicz norms but these formulations can be seen to be equivalent due to the elementary estimates
\begin{equation} \label{eq:equi2}
\frac{1}{2} \max\big\{ \|f\|_{L_2(E)}^2, \mathrm{Ent}\big( \|f\|_E^2\big) \big\} \leq \|f\|_{L_2(\log L)(E)}^p \leq 14 \max\big\{ \|f\|_{L_2(E)}^2, \mathrm{Ent}\big( \|f\|_E^2\big) \big\},
\end{equation}
which hold for every function $f:\Omega\to E$. We refer to Section \ref{sec:entropy} for more on this topic.


\subsection{Distortion lower bounds for quotients}
If $(\MM,d)$ and $(\NN,\delta)$ are two metric spaces,  the bi-Lipschitz distortion of $\MM$ into $\NN$, denoted by $\msf{c}_\NN(\MM)$, is the least $D\geq1$ for which there exists $\sigma\in(0,\infty)$ and a mapping $f:\MM\to\NN$ satisfying
\begin{equation}
\forall \ x,y\in\MM,\qquad \ \sigma d(x,y) \leq \delta(f(x),f(y)\big)\leq \sigma D d(x,y).
\end{equation}
Let $\ms{C}_n$ be the hypercube equipped with the Hamming metric $\rho(x,y) = \|x-y\|_1$.  A geometrically important application of the vector-valued Poincar\'e inequality \eqref{eq:ivv} is that it yields optimal distortion lower bounds for embeddings of $(\ms{C}_n,\rho)$ into normed spaces with prescribed Rademacher type.  Recall that $(E,\|\cdot\|_E)$ has Rademacher type $p\in[1,2]$ with constant $T\in(0,\infty)$ if for every $n\in\N$ and every vectors $x_1,\ldots,x_n\in E$, we have
\begin{equation}
\rmint_{\ms{C}_n} \Big\| \sum_{i=1}^n \e_i x_i\Big\|_E^p \,\diff\sigma_n(\e) \leq T^p \sum_{i=1}^n \|x_i\|_E^p.
\end{equation}
A classical argument due to Enflo \cite{Enf69,Enf78} combined with \cite{IVV20} shows that when a normed space $E$ has Rademacher type $p$ with constant $T$, then $\msf{c}_E(\ms{C}_n) \geq cT^{-1} n^{1-1/p}$, where $c>0$ is universal.

Let $\ms{R}\subseteq \ms{C}_n\times\ms{C}_n$ be an arbitrary equivalence relation and denote by $\ms{C}_n/\ms{R}$ the set of all equivalence classes of $\ms{R}$ equipped with the quotient metric, which is given by
\begin{equation}
\forall \ [\e],[\delta]\in\ms{C}_n/\ms{R}, \quad \rho_{\ms{C}_n/\ms{R}} \big( [\e], [\delta] \big) \eqdef \min\big\{ \rho(\eta_1,\zeta_1)+\cdots+\rho(\eta_{k},\zeta_k)\big\};
\end{equation}
here the minimum is taken over all $k\geq1$ and $\eta_1,\ldots,\eta_k,\zeta_1,\ldots,\zeta_k\in\ms{C}_n$ with $\eta_1\in[\e]$, $\zeta_k\in[\delta]$ and $[\zeta_j] = [\eta_{j+1}]$ for every $j\in\{1,\ldots,k-1\}$.  We will also denote by $\partial_i \ms{R}$ the boundary of $\ms{R}$ in the direction $i$, that is
\begin{equation} \label{eq:boundaries}
\forall \ i\in\{1,\ldots,n\}, \quad \partial_i \ms{R} \eqdef \big\{ \e\in\ms{C}_n: \ \big(\e, (\e_1,\ldots,\e_{i-1},-\e_i,\e_{i+1},\ldots,\e_n)\big) \notin \ms{R}\big\}
\end{equation} Enflo's argument combined with the results of \cite{IVV20} yields the general lower bound 
\begin{equation} \label{eq:lb-quotient}
\msf{c}_E\big(\ms{C}_n/\ms{R}\big) \geq \frac{cT^{-1} \big\| \rho_{\ms{C}_n/\ms{R}} ([\e],[\delta]) \big\|_{L_p(\ms{C}_n\times\ms{C}_n)}}{\big( \sum_{i=1}^n \sigma_n(\partial_i\ms{R}) \big)^{1/p}}
\end{equation}
for the bi-Lipschitz distortion of the quotient $\ms{C}_n/\ms{R}$ into a normed space of Rademacher type $p$, where $c>0$ is a universal constant.  As a consequence of the new vector-valued logarithmic Sobolev inequalities, we will prove the following improved bound in Section \ref{sec:distortion} below.

\begin{corollary} \label{cor:distortion}
There exists a universal constant $c>0$ such that if a normed space $E$ has Rademacher type $p$ with constant $T$, then for any $n\in\N$, we have
\begin{equation}
\msf{c}_E\big(\ms{C}_n/\ms{R}\big) \geq \frac{cT^{-1} \big\| \rho_{\ms{C}_n/\ms{R}} ([\e],[\delta]) \big\|_{L_p(\log L)^{p/2}(\ms{C}_n\times\ms{C}_n)}}{\big( \sum_{i=1}^n \sigma_n(\partial_i\ms{R}) \big)^{1/p}}.
\end{equation} 
\end{corollary}

A different (incomparable) refinement of the lower bound \eqref{eq:lb-quotient} has been obtained by means of the vector-valued influence inequalities in \cite[Section~8]{CE22}.  Similar lower bounds can be obtained from metric logarithmic Sobolev inequalities in full generality including, for instance,  for quotients of the symmetric group following the result of Theorem \ref{thm:symmetric}. 


\subsection*{Acknowledgements} We are grateful to Michel Ledoux and Ramon van Handel for valuable discussions and constructive feedback.


\section{Preliminary discussion} \label{sec:pre}

\subsection{Historical interlude} Before demonstrating the ideas leading to Theorem \ref{thm:main},  we digress to discuss the existing approaches to $L_p$ logarithmic Sobolev inequalities in discrete and continuous settings.  The case $p=2$ of inequality \eqref{eq:talagrand}, which is equivalent to Bonami's hypercontractive inequality \cite{Bon70}, is quite singular in that it can be proven in many distinct (and simple) ways which do not extend to other values of $p$ so we shall purposefully not address it separately.

In \cite{Led88},  Ledoux initiated the study of $L_p$ logarithmic Sobolev inequalities by proving the exact analogue of \eqref{eq:talagrand} in Gauss space $(\R^n,\gamma_n)$, where $\gamma_n$ is the standard normal distribution.  His work thus improves upon the dimension-free $L_p$ Poincar\'e inequality in Gauss space which had previously been proven by Maurey and Pisier in \cite{Pis86}.  Ledoux's proof relies on the coarea formula and crucially uses the Gaussian isoperimetric inequality \cite{ST74,Bor75}.  As Ledoux himself points out, his theorem was only new in the endpoint case $p=1$.  For $p\in(1,\infty)$, the corresponding inequality also follows by a concatenation of Meyer's Riesz transform inequalities in Gauss space \cite{Mey84} (see also \cite{Pis88}) and a delicate result of Bakry and Meyer \cite{BM82} on the boundedness of operators of the form $\ms{L}^{-\alpha}$ in Orlicz spaces, where $\ms{L}$ is the generator of a hypercontractive semigroup (see also \cite{Fei75}).  

Moreover, Ledoux mentions in \cite{Led88} that his results can be extended to a vector-valued setting (in the spirit of \cite{Pis86}) but his proof never appeared in print.  His approach (see Remark \ref{rem:gauss}) in fact gives a bound for $\mr{Ent}(\|f\|_E^2)$ which,  when combined with \eqref{eq:equi2}, yields a Gaussian version of Theorem \ref{thm:main} for $p=2$.  In turn, this entropy bound is obtained by applying the scalar logarithmic Sobolev inequality of \cite{Gro75} to the function $\|f\|_E$ and using the chain rule to differentiate the norm.  Overcoming the use of a chain rule (which does not exist in the discrete setting) is the main novelty of Theorem \ref{thm:main}.
 
Talagrand's technique yielding \eqref{eq:talagrand} on the hypercube is reminiscent of Ledoux's approach to the inequality's Gaussian counterpart. Indeed,  the first step of his proof is an intricate induction which implies an isoperimetric-type inequality \cite[Theorem~1.1]{Tal93} for subsets of the discrete hypercube (see also \cite{EG22} for a recent refinement).  Even though an immediate analogue of the coarea formula does not exist in discrete spaces,  in \cite[Section~5]{Tal93} Talagrand uses a discrete layer cake representation and suitable applications of his discrete isoperimetric inequality to deduce \eqref{eq:talagrand}.  Similarly to the Gaussian case,  a different proof of the $L_p$ logarithmic Sobolev inequality \eqref{eq:talagrand} for $p\in(1,\infty)$ can be obtained using Lust-Piquard's dimension-free bounds for the Riesz transform on the hypercube \cite{LP98}.  It is worth emphasizing that the argument leading to \eqref{eq:talagrand} only requires a lower bound for the Riesz transform which holds for every $p\in(1,\infty)$ as shown in \cite[Theorem~1.1]{BELP08}, whereas the corresponding upper bound only holds for $p\geq2$ (see \cite[p.~283]{LP98}).  To the best of our knowledge these are the only two known proofs to the logarithmic Sobolev inequality \eqref{eq:talagrand}.

We note in passing that a different refinement of the $L_2$ Poincar\'e inequality is Talagrand's influence inequality obtained in \cite{Tal94}.  The original proof of the influence inequality relied on Parseval's identity for the Fourier--Walsh transform along with a clever layer cake decomposition and a use of hypercontractivity, whereas a semigroup proof was later obtained in \cite{CL12}.  In \cite{CE22}, which is a precursor to this work,  we investigated vector-valued versions of Talagrand's influence inequality, obtained $L_p$ analogues and extensions of it and studied its geometric consequences. Unlike Theorem \ref{thm:main}, all the vector-valued influence inequalities of \cite{CE22} have either suboptimal assumptions on the target space\mbox{ (e.g.~some of them assume superreflexivity) or suboptimal exponents in the Orlicz spaces.}


\subsection{About the proof}  \label{sec:about-the-proof}
Unlike the influence inequality which has been successfully generalized to vector-valued functions, the two available arguments for Talagrand's inequality \eqref{eq:talagrand} appear to be limited to the scalar-valued case.  Indeed, the original proof relying on a discrete isoperimetric inequality stumbles upon the conceptual problem of understanding how to create a layer cake representation for vector-valued functions.  On the other hand,  finding Banach space-valued extensions of Lust-Piquard's Riesz transform inequalities \cite{LP98,BELP08} has been a long-standing open problem in harmonic analysis which appears intractable with current methods (see the discussion in \cite[Section~4]{Nao16}).  Nevertheless, an argument inspired by \cite{CL12, CE22} can yield inequalities similar to those of Theorem \ref{thm:main}, alas with a suboptimal Orlicz space on the\mbox{ left-hand side. This is further explained in Section \ref{sec:suboptimal}.}

Surprisingly, our proof of the optimal vector-valued logarithmic Sobolev inequality \eqref{eq:main} does not yield a new approach to the scalar inequality \eqref{eq:talagrand} but rather uses it \emph{as a black box}.  The starting point is the observation that the Orlicz norm on the left-hand side of \eqref{eq:main} can be controlled by a sum of two terms,  the first being an entropy-like term which captures the oscillations of $\|f-\mb{E}f\|_E$ and the second being further dominated by the $L_1$ norm of $f-\mb{E}f$.  The latter summand can thus be controlled by the vector-valued $L_1$ Poincar\'e inequality \eqref{eq:ivv} of \cite{IVV20}.  A direct application of Talagrand's inequality \eqref{eq:talagrand} is insufficient to control the remaining (scalar) Orlicz term by the vector-valued gradient of $f$.  Instead,  we apply an inequality which is a technical step in Talagrand's proof of \eqref{eq:talagrand}, which crucially controls said Orlicz norm by an \emph{asymmetric} gradient of the norm of $f$. This finally allows us to get the desired \mbox{bound via a differentiation argument, inspired by Ledoux's approach.}


\section{Proof of Theorem \ref{thm:main}}

The proof relies on the scalar $L_p$ logarithmic Sobolev inequality \eqref{eq:talagrand} of \cite{Tal93}.  In fact, rather than the inequality itself, we shall use a stronger statement which appeared in Talagrand's proof and is a similar inequality with respect to an {\it asymmetric} gradient. 
For a function $h:\ms{C}_n\to\R$,  we shall consider the asymmetric gradient $\msf{M}h:\ms{C}_n\to\R_+$ of $h$ that is given by
\begin{equation}
\forall \ x\in\ms{C}_n, \qquad \msf{M}h(x) = \left(\sum_{i=1}^n \partial_i h(x)_+^2\right)^{1/2},
\end{equation}
where $a_+ = \max\{a,0\}$ for a real number $a\in\R$.  Crucially for his proof of \eqref{eq:talagrand}, Talagrand proved the following refined technical result which appears in \cite[Proposition~5.1]{Tal93}.

\begin{proposition} [Talagrand] \label{prop:talagrand}
For every $p\in[1,\infty)$, there exists $\kappa_p\in(0,\infty)$ such that the following holds for any $n\in\N$.  If $h:\ms{C}_n\to\R_+$ is a nonnegative function satisfying $\sigma_n\{h=0\} \geq\tfrac{1}{2}$, then
\begin{equation}
\|h\|_{L_p(\log L)^{p/2}} \leq \kappa_p \big\| \msf{M}h\big\|_{L_p}.
\end{equation} 
\end{proposition}

Recall that any function $f:\ms{C}_n\to E$ admits a unique Fourier--Walsh expansion of the form
\begin{equation}
\forall \ x\in\ms{C}_n, \qquad f(x) = \sum_{S\subseteq\{1,\ldots,n\}} \hat{f}(S) w_S(x),
\end{equation}
where the Fourier coefficients $\hat{f}(S)\in E$ and the Walsh functions $w_S:\ms{C}_n\to\{-1,1\}$ are given by $w_S(x) = \prod_{i\in S} x_i$.  By the expansion above, it is clear that any function $f:\ms{C}_n\to E$ on the hypercube admits a unique multilinear extension on $\R^n$, that is, there exists a unique function $F:\R^n\to E$ such that $F|_{\ms{C}_n}=f$ and $F$ is affine in each variable.  

We shall use the following pointwise estimate which is a consequence of convexity.

\begin{lemma} \label{lem:M-control}
Let $(E,\|\cdot\|_E)$ be a normed space. Then for any $n\in\N$, every function $f:\ms{C}_n\to E$ satisfies the pointwise estimate
\begin{equation} \label{eq:M-control}
\forall \ x\in\ms{C}_n, \qquad \msf{M}\|f\|_E(x) \leq \Big( \rmint_{\ms{C}_n} \Big\| \sum_{i=1}^n \delta_i \partial_i f(x) \Big\|_E^2\Big)^{1/2}.
\end{equation}
\end{lemma}

\begin{proof}
Fix $x\in\ms{C}_n$.  Without loss of generality, we can always assume that $E$ is finite-dimensional as $F$ takes values in $\mr{span}\{F(x): x\in\R^n\}=\mr{span}\{f(x): x\in\ms{C}_n\}$ which has dimension at most $2^n$.  Therefore (by approximation), we can assume that $f(x)\neq0$ and that $\|\cdot\|_E$ is smooth on $E\setminus\{0\}$, that is,  for every $v\in E\setminus\{0\}$ there exists $\msf{D}_v^\ast \in E^\ast$ with $\|\msf{D}_{v}^*\|_{E^*}\leq 1$ such that if $\beta:(-\eta,\eta)\to E$ is a smooth function for some $\eta >0$ with $\beta(0)=v$, then
\begin{equation} \label{eq:norm-smooth}
\frac{\diff}{\diff t} \Big|_{t=0} \big\|\beta(t)\big\|_E = \Big\langle \msf{D}_{v}^*,  \frac{\diff}{\diff t}\Big|_{t=0} \beta(t)\Big\rangle.
\end{equation}
Fix also $i\in\{1,\ldots,n\}$ and let $F:\R^n\to E$ be the multilinear extension of $f$. The lemma will be a consequence of the following two simple technical claims.

\smallskip
\noindent {\bf Claim.} The function $\phi_i:\R\to\R_+$ given by
\begin{equation} \label{eq:phii}
\forall \ s\in\R, \qquad \phi_i(s) = \big\| F(x+se_i)\big\|_E
\end{equation}
is convex.

\smallskip

\noindent {\it Proof.} Indeed, by the multilinear expansion of $F$, we have
\begin{equation}
F(x+se_i) = \sum_{\substack{S\subset\{1\ldots,n\}: \\ i\notin S}} \hat{f}(S) w_S(x) +  (x_i+s) \sum_{\substack{S\subset\{1\ldots,n\}: \\ i\in S}} \hat{f}(S) w_{S\setminus\{i\}}(x),
\end{equation}
which is an affine function in the variable $s$.  Therefore,  $s\mapsto \|F(x+se_i)\|_E$ is indeed convex.
\hfill$\Box$

\smallskip

\noindent {\bf Claim.} We have the following pointwise estimate
\begin{equation} \label{eq:pointwise-i}
\partial_i\|f\|_E(x)_+ \leq \Big|\frac{\partial \|F\|_E}{\partial x_i}(x)\Big|.
\end{equation}

\smallskip

\noindent {\it Proof.} We have assumed that $\|\cdot\|_E$ is smooth on $E\setminus\{0\}$ and $f(x)\neq0$ so the right-hand side is well-defined.  We will distinguish two cases and include all the details for completeness.

\smallskip

\noindent {\it Case 1.} Suppose that $x_i=1$.  Then,  we will show that for every $h>0$, 
\begin{equation} \label{eq:case-1-lem}
\partial_i\|f\|_E(x)_+  = \left( \frac{\|F(x)\|_E - \|F(x-2e_i)\|_E}{2} \right)_+ \leq \left(\frac{\|F(x+he_i)\|_E-\|F(x)\|_E}{h}\right)_+
\end{equation} 
and the conclusion follows by taking a limit of $h\to0^+$.  Phrased in terms of the function $\phi_i$ of \eqref{eq:phii}, inequality \eqref{eq:case-1-lem} can be rewritten as
\begin{equation}
\forall \ h>0,\qquad \left(\frac{\phi_i(0)-\phi_i(-2)}{2} \right)_+ \leq \left(\frac{\phi_i(h)-\phi_i(0)}{h}\right)_+.
\end{equation}
If $\phi_i(0)\leq\phi_i(-2)$, then the inequality holds trivially.  On the other hand, if $\phi_i(0)>\phi_i(-2)$, we get from the convexity of $\phi_i$ that
\begin{equation}
\phi_i(0) \leq \frac{2}{2+h}\phi_i(h) + \frac{h}{2+h} \phi_i(-2),
\end{equation}
which can be rewritten as 
\begin{equation}
\frac{\phi_i(0)-\phi_i(-2)}{2} \leq \frac{\phi_i(h)-\phi_i(0)}{h} \leq \left(\frac{\phi_i(h)-\phi_i(0)}{h}\right)_+.
\end{equation}

\smallskip 

\noindent {\it Case 2.} Suppose that $x_i=-1$. Similarly, we will show that for every $h>0$,
\begin{equation}
\partial_i\|f\|_E(x)_+  = \left( \frac{\|F(x)\|_E - \|F(x+2e_i)\|_E}{2} \right)_+ \leq \left(\frac{\|F(x-he_i)\|_E - \|F(x)\|_E}{h}\right)_+.
\end{equation}
As before,  we can assume that $\phi_i(0)>\phi_i(2)$ and use convexity in the form of
\begin{equation}
\phi_i(0) \leq \frac{2}{2+h}\phi_i(-h) + \frac{h}{2+h} \phi_i(2),
\end{equation}
which can be rewritten as
\begin{equation*}
\frac{\phi_i(0)-\phi_i(2)}{2} \leq \frac{\phi_i(-h)-\phi_i(0)}{h} \leq \left(\frac{\phi_i(-h)-\phi_i(0)}{h}\right)_+. \qedhere
\end{equation*}

Combining the proven claim \eqref{eq:pointwise-i} with \eqref{eq:norm-smooth},  we get that
\begin{equation}
\begin{split}
\msf{M}\|f\|_E(x)^2 & = \sum_{i=1}^n \partial_i\|f\|_E(x)_+^2 \stackrel{\eqref{eq:pointwise-i}}{\leq} \sum_{i=1}^n \Big( \frac{\partial \|F\|_E}{\partial x_i}(x) \Big)^2 = \rmint_{\ms{C}_n} \Big| \sum_{i=1}^n \delta_i \frac{\partial \|F\|_E}{\partial x_i}(x) \Big|^2 \,\diff\sigma_n(\delta)
\\ & \stackrel{\eqref{eq:norm-smooth}}{=}  \rmint_{\ms{C}_n} \Big| \sum_{i=1}^n \delta_i \Big\langle \msf{D}_{f(x)}^*,\frac{\partial F}{\partial x_i}(x) \Big\rangle \Big|^2 \,\diff\sigma_n(\delta) = \rmint_{\ms{C}_n} \Big| \Big\langle \msf{D}_{f(x)}^*,\sum_{i=1}^n \delta_i \frac{\partial F}{\partial x_i}(x) \Big\rangle \Big|^2 \,\diff\sigma_n(\delta) 
\\ & \leq \rmint_{\ms{C}_n} \Big\| \sum_{i=1}^n \delta_i \frac{\partial F}{\partial x_i}(x) \Big\|_E^2\,\diff\sigma_n(\delta).
\end{split}
\end{equation}
Now notice that as $x\in\ms{C}_n$, the multilinearity of $F$ easily shows that
\begin{equation}
\frac{\partial F}{\partial x_i}(x) = x_i \partial_i f(x)
\end{equation}
and thus we can rewrite the above bound as
\begin{equation}
\msf{M}\|f\|_E(x)^2 \leq \rmint_{\ms{C}_n} \Big\| \sum_{i=1}^n x_i \delta_i \partial_if(x) \Big\|_E^2\,\diff\sigma_n(\delta) =  \rmint_{\ms{C}_n} \Big\| \sum_{i=1}^n \delta_i \partial_if(x) \Big\|_E^2\,\diff\sigma_n(\delta),
\end{equation}
where in the last equality we used that the random vectors $\delta$ and $(x_1\delta_1,\ldots,x_n\delta_n)$ have the same distribution when $\delta$ is distributed according to $\sigma_n$ and $x\in\ms{C}_n$ is fixed.  \hfill$\Box$
\end{proof}

Equipped with the above, we can prove Theorem \ref{thm:main}.

\begin{proof} [Proof of Theorem \ref{thm:main}]
Let $f:\ms{C}_n\to E$ be any function, where $E$ is a space of finite cotype,  and without loss of generality assume that $\mb{E} f=0$. Moreover,  consider the nonnegative function $h\eqdef \|f\|_E: \ms{C}_n\to\R_+$ and denote by $\msf{m}\geq0$ a median of $h$.  Then, the function $(h-\msf{m})_+$ satisfies $\sigma_n\{ (h-\msf{m})_+ =0 \} = \sigma_n\{h\leq \msf{m}\} \geq\tfrac{1}{2}$. Therefore, Proposition \ref{prop:talagrand} yields the estimate
\begin{equation} \label{eq:use-talagrand}
\big\|(h-\msf{m})_+\big\|_{L_p(\log L)^{p/2}} \leq \kappa_p \big\| \msf{M}(h-\msf{m})_+\big\|_{L_p}.
\end{equation}
Since $\msf{m}$ is a median of $h$,  we have
\begin{equation} \label{eq:use-markov}
\msf{m} \leq \frac{1}{\sigma_n\{h\geq \msf{m}\}} \rmint_{\{h\geq \msf{m}\}} h \,\diff\sigma_n \leq 2 \|h\|_{L_1} = 2 \|f\|_{L_1(E)},
\end{equation}
and thus combining the above with the inequalities $0\leq h\leq (h-\msf{m})_+ + \msf{m}$, we get
\begin{equation} \label{eq:used-Talagrand}
\begin{split}
\|f\|_{L_p(\log L)^{p/2}(E)} = \|h\|_{L_p(\log L)^{p/2}}  \leq \big\|& (h-\msf{m})_+ + \msf{m}\big\|_{L_p(\log L)^{p/2}} 
\\ & \stackrel{\eqref{eq:use-talagrand} \wedge \eqref{eq:use-markov}}{\leq} \kappa_p  \big\| \msf{M}(h-\msf{m})_+\big\|_{L_p} + 2 \|f\|_{L_1(E)}.
\end{split}
\end{equation}
\smallskip
{\bf Claim.} For any scalar function $\phi:\ms{C}_n\to \R$,  we have the pointwise estimate
\begin{equation} \label{eq:Mh+}
\forall \ x\in\ms{C}_n,\qquad \msf{M}\phi_+(x) \leq \msf{M}\phi(x).
\end{equation}
{\it Proof.} Since $\phi_+(x) = \max\big\{\phi(x),0\big\}$, it suffices to show that for each $x,y\in\ms{C}_n$,
\begin{equation}
\big( \phi_+(x) - \phi_+(y)\big)_+ \leq \big( \phi(x)- \phi(y)\big)_+.
\end{equation}
Denoting by $a=\phi(x)$ and $b=\phi(y)$, the inequality can be rewritten as $(a_+-b_+)_+ \leq (a-b)_+$. If $a_+\leq b_+$ the inequality is trivially true so we will assume that $a_+>b_+$, which implies $a>0$.  In this case, if $b\leq0$, then it amounts to $a\leq a-b$ which holds and if $b>0$ it becomes an identity. \hfill$\Box$

\smallskip

Combining \eqref{eq:used-Talagrand} with inequality \eqref{eq:Mh+}, we thus deduce that
\begin{equation} \label{eq:main-two-terms}
\|f\|_{L_p(\log L)^{p/2}(E)} \leq \kappa_p  \big\| \msf{M}h\big\|_{L_p} + 2 \|f\|_{L_1(E)},
\end{equation}
since $\msf{M}(h-m) \equiv \msf{M}h$. Moreover, by the $L_1$ Poincar\'e inequality \eqref{eq:ivv} of Ivanisvili, van Handel and Volberg \cite{IVV20},  the second term can be further upper bounded as
\begin{equation}
\|f\|_{L_1(E)} \leq \msf{C}_1(E) \rmint_{\ms{C}_n} \Big\|\sum_{i=1}^n \delta_i \partial_i f\Big\|_{L_1(E)} \,\diff\sigma_n(\delta) \leq \msf{C}_1(E) \left( \rmint_{\ms{C}_n} \Big\|\sum_{i=1}^n \delta_i \partial_i f\Big\|_{L_p(E)}^p \,\diff\sigma_n(\delta)\right)^{1/p},
\end{equation}
where we used that $E$ has finite cotype.   Finally, to control the first term in \eqref{eq:main-two-terms} we use the pointwise inequality of Lemma \ref{lem:M-control} and Kahane's inequality with sharp constant \cite{Kah64,LO94} to obtain the estimate
\begin{equation} \label{eq:use-kahane}
\big\| \msf{M}h\big\|_{L_p} \leq \Big\| \Big( \rmint_{\ms{C}_n} \Big\| \sum_{i=1}^n \delta_i \partial_i f(x) \Big\|_E^2\Big)^{1/2}\Big\|_{L_p} \leq \sqrt{2} \left( \rmint_{\ms{C}_n} \Big\| \sum_{i=1}^n \delta_i \partial_i f\Big\|_{L_p(E)}^p \,\diff\sigma_n(\delta)\right)^{1/p}.
\end{equation}
This completes the proof of the theorem.
\end{proof}

The same arguments allow us to prove the refined Pisier inequality of Corollary \ref{cor:pisier}.

\begin{proof} [Proof of Corollary \ref{cor:pisier}]
Recall that Pisier's inequality \cite[Lemma~7.3]{Pis86} (see also \cite{HN13} for the value of the multiplicative constant) asserts that for any $r\in[1,\infty)$, every $f:\ms{C}_n\to E$ satisfies
\begin{equation}
\|f-\mb{E}f\|_{L_r(E)} \leq (\log n+1) \left( \rmint_{\ms{C}_n} \Big\| \sum_{i=1}^n \delta_i \partial_i f\Big\|_{L_r(E)}^r \,\diff\sigma_n(\delta)\right)^{1/r}
\end{equation}
Assuming again that $\mb{E}f=0$, \eqref{eq:main-two-terms} combined with \eqref{eq:use-kahane} and Pisier's inequality for $r=1$ gives the desired upper bound in a general normed space.
\end{proof}

\begin{remark} \label{rem:general}
The above argument combined with the main result of Ivanisvili, van Handel and Volberg \cite{IVV20} yields a mutual refinement of \eqref{eq:main} and \eqref{eq:strong-pisier} for general Banach space targets.  Following \cite{IVV20},  let $\xi_1(t),\ldots,\xi_n(t)$ be i.i.d.~random variables satisfying
\begin{equation}
\mb{P}\big\{ \xi_i(t)=1 \big\} = \frac{1+e^{-t}}{2} \qquad \mbox{and} \qquad \mb{P}\big\{ \xi_i(t) = -1\big\} = \frac{1-e^{-t}}{2}
\end{equation}
and denote by $\delta_i(t) \eqdef \frac{\xi_i(t)-e^{-t}}{\sqrt{1-e^{-2t}}}$.  In \cite[Theorem~1.4]{IVV20}, the authors showed the estimate
\begin{equation}
\|f-\mb{E}f\|_{L_r(E)} \leq \frac{\pi}{2} \rmint_0^\infty \left( \mb{E} \Big\| \sum_{i=1}^n \delta_i(t) \partial_if \Big\|_{L_r(E)}^r\right)^{1/r} \frac{\diff t}{\sqrt{e^{2t}-1}}.
\end{equation}
Plugging this bound for $r=1$ and \eqref{eq:use-kahane}  in \eqref{eq:main-two-terms} yields the stronger inequality
\begin{equation} \label{eq:most-general}
\begin{split}
\big\|f-\mb{E}f\big\|_{L_p(\log L)^{p/2}(E)} \leq \kappa_p & \left( \rmint_{\ms{C}_n} \Big\| \sum_{i=1}^n \delta_i \partial_i f\Big\|_{L_p(E)}^p \,\diff\sigma_n(\delta)\right)^{1/p} 
\\ &+ \pi  \rmint_0^\infty \mb{E} \Big\| \sum_{i=1}^n \delta_i(t) \partial_if \Big\|_{L_1(E)} \, \frac{\diff t}{\sqrt{e^{2t}-1}}.
\end{split}
\end{equation}
\end{remark}

\begin{remark} \label{rem:gauss}
A Gaussian version of Theorem \ref{thm:main} was announced in Ledoux's classical work \cite{Led88} yet the proof never appeared in print.  In personal communications,  Ledoux confirmed to us that his claim in \cite{Led88} was the following. For any normed space $E$ and any $n\in\N$, every smooth function $f:(\R^n,\gamma_n)\to E$ satisfies the inequality
\begin{equation} \label{eq:ledoux}
\mathrm{Ent}\big(\|f\|_E^2\big) \leq 2 \rmint_{\R^n} \Big\| \sum_{i=1}^n g_i \partial_i f\Big\|_{L_2(E)} \, \diff\gamma_n(g),
\end{equation}
where the entropy and $L_2(E)$ norm are both with respect to the standard Gaussian measure $\gamma_n$. To see this, observe that Gross' logarithmic Sobolev inequality applied to the function $\|f\|_E$ yields
\begin{equation}
\mathrm{Ent}\big(\|f\|_E^2\big) \leq 2 \sum_{i=1}^n \big\|\partial_i \|f\|_E\big\|_{L_2}^2 = 2 \rmint_{\R^n} \Big\| \sum_{i=1}^n g_i \partial_i \|f\|_E \Big\|_{L_2}^2\,\diff\gamma_n(g).
\end{equation}
Arguing as in the proof of Theorem \ref{thm:main}, we can assume without loss of generality that $\|\cdot\|_E$ is smooth on $E\setminus\{0\}$ and that $f(x)\neq 0$ almost surely,  so that we have $\partial_i \|f\|_E(x) = \langle \msf{D}^\ast_{f(x)}, \partial_i f(x)\rangle$ for some linear functional $\msf{D}^\ast_{f(x)}$ in the unit ball of $E^\ast$. Therefore, 
\begin{equation}
\mathrm{Ent}\big(\|f\|_E^2\big) \leq 2 \rmint \!\!\!\!\!\rmint_{\R^n\times\R^n} \Big\langle \msf{D}^\ast_{f(x)}, \sum_{i=1}^n g_i\partial_if(x)\Big\rangle^2\,\diff\gamma_{2n}(x,g) \leq 2 \rmint_{\R^n} \Big\| \sum_{i=1}^n g_i \partial_i f\Big\|_{L_2(E)}^2 \, \diff\gamma_n(g).
\end{equation}
The elementary inequality \eqref{eq:equi2} (see also Section \ref{sec:entropy}) combined with Ledoux's inequality \eqref{eq:ledoux} and the vector-valued Gaussian Poincar\'e inequality
\begin{equation}
\big\|f-\mb{E}f\big\|_{L_2} \leq \frac{\pi}{2} \left(\rmint_{\R^n} \Big\|\sum_{i=1}^n g_i \partial_if\Big\|_{L_2(E)}^2 \,\diff\gamma_n(g) \right)^{1/2}
\end{equation}
of Maurey and Pisier \cite[Corollary~2.4]{Pis86} implies the Gaussian version of Theorem \ref{thm:main} for $p=2$ and functions with values in an arbitrary normed space $E$. An extension of this estimate to all values of $p\in[1,\infty)$ follows from the general inequality \eqref{eq:most-general} by standard considerations involving the central limit theorem, see e.g.~\cite[Remark~1.3]{IVV20}.

It is worth emphasizing that, as there is no chain rule for the discrete derivatives on the hypercube,  the main additional difficulty arising in the proof of Theorem \ref{thm:main} was the need to relate discrete partial derivatives of $\|f\|_E$ with classical partial derivatives of $\|F\|_E$. Each $i$-th partial derivative turns out to only be controlled by a one-sided inequality on an $(n-1)$-dimensional subcube of $\ms{C}_n$ depending on $i$  (Lemma \ref{lem:M-control}) yet this is sufficient to bound the full gradient in view of Talagrand's asymmetric logarithmic Sobolev inequality (Proposition \ref{prop:talagrand}).
\end{remark}


\section{Self-improvement of Poincar\'e inequalities} \label{sec:boost}

In this Section we formalize the discussion of Section \ref{sec:intro-boost}, where we claimed that metric Poincar\'e inequalities for norms can be boosted to logarithmic Sobolev-type inequalities provided that the corresponding scalar-valued inequality holds.  More precisely,  we prove the following proposition.

\begin{proposition} \label{prop:boost}
Let $p\geq1$ and $\alpha>0$.  Fix a normed space $(E,\|\cdot\|_E)$, a probability space $(\Omega,\mu)$ and a measure $\beta$ on $\Omega\times\Omega$.  Suppose that every bounded function $f:\Omega\to E$ satisfies
\begin{equation} \label{eq:vec-pi}
\big\|f-\mb{E} f\big\|_{L_p(E)}^p \leq \rmint\!\!\!\!\!\rmint_{\Omega\times\Omega} \big\|f(x)-f(y)\big\|_E^p \,\diff\beta(x,y)
\end{equation}
and that every scalar bounded function $h:\Omega \to \C$ satisfies
\begin{equation} \label{eq:sc-lsi}
\big\|h-\mb{E} h\big\|_{L_p(\log L)^\alpha}^p \leq \msf{C}^p \rmint\!\!\!\!\!\rmint_{\Omega\times\Omega} \big|h(x)-h(y)\big|^p \,\diff\beta(x,y)
\end{equation}
for some $\msf{C}>0$. Then, for every bounded function $f:\Omega\to E$ we have
\begin{equation}
\big\|f-\mb{E} f\big\|_{L_p(\log L)^\alpha(E)}^p \leq 2^{p-1}\big( \msf{C}^p+1)\rmint\!\!\!\!\!\rmint_{\Omega\times\Omega} \big\|f(x)-f(y)\big\|_E^p \,\diff\beta(x,y).
\end{equation}
\end{proposition}

\begin{proof}
Assume that $\mb{E} f=0$ and denote by $h=\|f\|_E$. Then, we have
\begin{equation} \label{eq:apply-tr}
\|f\|_{L_p(\log L)^\alpha(E)}^p = \|h\|_{L_p(\log L)^\alpha}^p \leq 2^{p-1}\big( \big\|h-\mb{E}h\big\|_{L_p(\log L)^\alpha}^p + \big| \mb{E}h\big|^p \big)
\end{equation}
by the triangle inequality for Orlicz norms and H\"older's inequality $(A+B)^p \leq 2^{p-1}(A^p+B^p)$. To control the first summand in \eqref{eq:apply-tr},  we use \eqref{eq:sc-lsi} and get
\begin{equation}
\begin{split}
\big\|h&-\mb{E}h\big\|_{L_p(\log L)^\alpha}^p \stackrel{\eqref{eq:sc-lsi}}{\leq} \msf{C}^p \rmint\!\!\!\!\!\rmint_{\Omega\times\Omega} \big|h(x)-h(y)\big|^p \,\diff\beta(x,y) 
\\ & = \msf{C}^p \rmint\!\!\!\!\!\rmint_{\Omega\times\Omega} \big|\|f(x)\|_E-\|f(y)\|_E\big|^p \,\diff\beta(x,y) \leq  \msf{C}^p \rmint\!\!\!\!\!\rmint_{\Omega\times\Omega} \big\|f(x)-f(y)\big\|_E^p \,\diff\beta(x,y).
\end{split}
\end{equation}
For the second summand,  observe that \eqref{eq:vec-pi} gives
\begin{equation}
\big| \mb{E}h\big|^p = \big| \mb{E}\|f\|_E\big|^p \leq \mb{E}\|f\|_E^p \stackrel{\eqref{eq:vec-pi}}{\leq} \rmint\!\!\!\!\!\rmint_{\Omega\times\Omega} \big\|f(x)-f(y)\big\|_E^p \,\diff\beta(x,y)
\end{equation}
and this concludes the proof.
\end{proof}

As an example, combining Proposition \ref{prop:boost} with \cite{Bon70,Gro75,IVV20}, we easily deduce a metric logarithmic Sobolev inequality on the hypercube under Rademacher type assumptions. The same estimate also follows directly from \eqref{eq:most-general} but the proof via Proposition \ref{prop:boost} is substantially simpler.

\begin{corollary} \label{cor:type}
For every $p\in[1,2]$ and every normed space $(E,\|\cdot\|_E)$ of Rademacher type $p$, there exists $\tau_p(E)\in(0,\infty)$ such that for any $n\in\N$, every $f:\ms{C}_n\to E$ satisfies
\begin{equation}
\big\|f-\mb{E}f\big\|_{L_p(\log L)^{p/2}(E)} \leq \tau_p(E) \left(\sum_{i=1}^n \big\|\partial_i f\big\|_{L_p(E)}^p\right)^{1/p}.
\end{equation}
Moreover, $\tau_p(E)$ is proportional to the Rademacher type $p$ constant of $E$.
\end{corollary}


\subsection{Distortion lower bounds for quotients} \label{sec:distortion} Using Corollary \ref{cor:type}, we will prove Corollary \ref{cor:distortion}.

\begin{proof} [Proof of Corollary \ref{cor:distortion}]
Let $\ms{R}\subseteq\ms{C}_n\times\ms{C}_n$ be an equivalence relation and consider an embedding $f:\ms{C}_n/\ms{R} \to E$ satisfying the bi-Lipschitz condition
\begin{equation} \label{lip}
\forall \ [\e],[\delta]\in\ms{C}_n/\ms{R},\qquad \sigma \rho_{\ms{C}_n/\ms{R}}\big( [\e],[\delta]\big) \leq \big\|f\big([\e]\big)-f\big([\delta]\big)\big\|_E \leq \sigma D\rho_{\ms{C}_n/\ms{R}}\big( [\e],[\delta]\big)
\end{equation}
for some $\sigma>0$ and $D\geq1$. Also, consider  the lift $F:\ms{C}_n\to E$ of $f$ given by $F(\e) = f([\e])$ for $\e\in\ms{C}_n$. Then, the lower Lipschitz condition on $f$, yields
\begin{equation}
\begin{split}
\big\|F(\e)-F(\delta)\big\|_{L_p(\log L)^{p/2}(\ms{C}_n\times\ms{C}_n;E)} & = \big\|f\big([\e]\big)-f\big([\delta]\big)\big\|_{L_p(\log L)^{p/2}(\ms{C}_n\times\ms{C}_n;E)}
\\ & \stackrel{\eqref{lip}}{\geq} \sigma  \big\| \rho_{\ms{C}_n/\ms{R}}\big( [\e],[\delta]\big) \big\|_{L_p(\log L)^{p/2}(\ms{C}_n\times\ms{C}_n)}.
\end{split}
\end{equation} 
On the other hand, by the triangle inequality and Corollary \ref{cor:type}, we have
\begin{equation*}
\begin{split}
\big\|F(\e)-F(\delta)\big\|_{L_p(\log L)^{p/2}(\ms{C}_n\times\ms{C}_n;E)} \leq 2 \big\|F - \mb{E}F\big\|_{L_p(\log L)^{p/2}(E)}
 \leq 2\tau_p(E) \left(\sum_{i=1}^n \big\|\partial_iF\big\|_{L_p(E)}^p\right)^{1/p}.
\end{split}
\end{equation*}
Now notice that for each $i$,
\begin{equation}
\begin{split}
\big\|\partial_i F\big\|_{L_p(E)}^p = \frac{1}{2^p} \rmint_{\ms{C}_n} \big\|f\big([\e]\big) - f\big([(\e_1,\ldots,-\e_i,\ldots,\e_n)]\big)\big\|_E^p\,\diff\sigma_n(\e)
 \stackrel{\eqref{lip}}{\leq} \sigma^p D^p  \sigma_n\big(\partial_i \ms{R}\big),
\end{split}
\end{equation}
as $\rho_{\ms{C}_n/\ms{R}} \big( [\e], [(\e_1,\ldots,-\e_i,\ldots,\e_n)]\big) = 2$ if $\e\in\partial_i\ms{R}$ and otherwise it is 0. Combining all the above readily gives the desired lower bound on $D$.
\end{proof}


\subsection{Proof of Theorem \ref{thm:symmetric}} 
Recall that a normed space $(E,\|\cdot\|_E)$ has martingale type 2 with constant $M\in(0,\infty)$ if for any $n\in\N$, any probability space $(\Omega,\ms{F},\mu)$ and any filtration $\{\ms{F}_i\}_{i=0}^n$ of sub-$\sigma$-algebras of $\ms{F}$, \mbox{every $E$-valued martingale $\{\ms{M}_i:\Omega\to E\}_{i=0}^n$ adapted to $\{\ms{F}_i\}_{i=0}^n$ satisfies}
\begin{equation} \label{eq:mtype}
\big\|\ms{M}_n-\ms{M}_0\big\|_{L_2(\mu;E)}^2 \leq M^2 \sum_{i=1}^n \big\| \ms{M}_i-\ms{M}_{i-1}\big\|^2_{L_2(\mu;E)}.
\end{equation}
The least constant $M>0$ for which \eqref{eq:mtype} is satisfied shall be denoted by $\msf{m}_2(E)$.

In \cite[Section~6]{KN06},  Khot and Naor proved lower bounds for the bi-Lipschitz distortion of metric spaces with prescribed length in the sense of Schechtman \cite{Sch82} into normed spaces of given martingale type.  An important example of such a metric space is the Cayley graph of the symmetric group $\ms{S}_n$ equipped with generating set $\ms{T}_n$ consisting of all transpositions, as was proven in influential earlier work of Maurey \cite{Mau79}.  For their proof of the distortion lower bounds, Khot and Naor use a Poincar\'e type inequality in such spaces.  In the case of the symmetric group, the following inequality can be extracted from the proof of \cite[Theorem~7]{KN06}.

\begin{theorem} [Khot--Naor] \label{thm:kn}
If $(E,\|\cdot\|_E)$ is a Banach space of martingale type 2, then for any $n\in\N$, every function $f:\ms{S}_n\to E$ satisfies
\begin{equation} \label{eq:kn}
\big\| f-\mb{E}f\big\|_{L_2(E)}^2 \leq \frac{2 \msf{m}_2(E)^2}{n-1} \sum_{\tau\in\ms{T}_n} \rmint_{\ms{S}_n} \big\|f(\pi) - f(\tau \circ\pi) \big\|_E^2\,\diff\nu_n(\pi) .
\end{equation}
\end{theorem}

To derive the vector-valued logarithmic Sobolev inequality of Theorem \ref{thm:symmetric}, we shall also need the following scalar-valued result of Diaconis and Saloff-Coste \cite[Section~4.3]{DSC96}.

\begin{theorem} [Diaconis--Saloff-Coste] \label{thm:dsc}
For any $n\in\N$, every function $f:\ms{S}_n\to\R$ satisfies
\begin{equation}
\mathrm{Ent}\big(f^2\big) \leq \frac{6\log n}{n-1} \sum_{\tau\in\ms{T}_n} \rmint_{\ms{S}_n} \big|f(\pi) - f(\tau\circ\pi)\big|^2\,\diff\nu_n(\pi).
\end{equation}
\end{theorem}

\begin{proof} [Proof of Theorem \ref{thm:symmetric}]
In view of the inequality \eqref{eq:equi2} (see also the more general Lemma \ref{lem:orlicz-ineq}) and the scalar-valued case of inequality \eqref{eq:kn}, Theorem \ref{thm:dsc} implies that every function $f:\ms{S}_n\to\R$ satisfies
\begin{equation}
\big\| f-\mb{E}f\big\|_{L_2(\log L)}^2 \leq \frac{C\log n}{n-1} \sum_{\tau\in\ms{T}_n} \rmint_{\ms{S}_n} \big|f(\pi) - f(\tau \circ\pi) \big|^2\,\diff\nu_n(\pi),
\end{equation}
where $C>0$ is a universal constant. Combining this with the vector-valued inequality of Theorem \ref{thm:kn} and the general implication of Proposition \ref{prop:boost}, we deduce Theorem \ref{thm:symmetric}.
\end{proof}


\subsection{Orlicz norms and $\alpha$-entropies} \label{sec:entropy}
In order to use Proposition \ref{prop:boost} in concrete examples (including Theorem \ref{thm:symmetric}), we needed the elementary estimate \eqref{eq:equi2} relating the $L_2(\log L)$ norm with entropy. In this section, we extend this two-sided bound to more general Orlicz norms. If $x\in(0,\infty)$ and $\alpha>0$, we use the ad hoc notation
\begin{equation}
\log^\alpha x = \sign(\log x) |\log x|^\alpha,
\end{equation}
so that $x\mapsto\log^\alpha x$ is increasing on $(0,\infty)$. Finally, for a nonnegative function $h:(\Omega,\mu)\to[0,\infty)$, and $\alpha>0$ we denote its \emph{$\alpha$-entropy} by
\begin{equation}
\mr{Ent}^\alpha(h) = \rmint_\Omega h \log^\alpha\Big( \frac{h}{\rmint_\Omega h\,\diff\mu}\Big) \,\diff\mu,
\end{equation}
where $0\log^\alpha0$ will always be understood to be 0. 

We shall need a two-sided estimate for Orlicz norms.  The special case $p=2$ and $\alpha=1$ is implicit in the work \cite[Section~4]{BG99b} of Bobkov and G\"otze.  We have been generous with the dependence of the constants on the various parameters involved.

\begin{lemma} \label{lem:orlicz-ineq}
Fix $p\geq1$ and $\alpha>0$. If $(\Omega,\mu)$ is a probability space, $(E,\|\cdot\|_E)$ is a normed space and $f:\Omega\to E$ is a measurable function, then we have
\begin{equation} \label{eq:orlicz-ineq}
c_\alpha^{-1} \max\big\{ \|f\|_{L_p(E)}^p, \mathrm{Ent}^\alpha\big( \|f\|_E^p\big) \big\} \leq \|f\|_{L_p(\log L)^\alpha(E)}^p \leq C_\alpha \max\big\{ \|f\|_{L_p(E)}^p, \mathrm{Ent}^\alpha\big( \|f\|_E^p\big) \big\},
\end{equation}
where $c_\alpha = 2^{(\alpha-1)_+}\big(1+(\alpha/e)^\alpha\big) + (\alpha/e)^\alpha$ and 
\begin{equation}
C_\alpha = \max\Big\{ \log^\alpha(e+1)+2 (\alpha/e)^\alpha, \tfrac{2^{1/\alpha}+e-1}{2^{1/\alpha}-1}\log^\alpha\Big(e+ \tfrac{2^{1/\alpha}+e-1}{2^{1/\alpha}-1}\Big) \Big\} + 2.
\end{equation}
\end{lemma}

\begin{proof}
Since all quantities involved depend only on $\|f\|_E$, we can assume without loss of generality that $f$ is scalar-valued and nonnegative. We start by proving the inequality on the left. Assume that $ \|f\|_{L_p(\log L)^\alpha}^p\leq 1$, that is,
\begin{equation}
\rmint_\Omega f^p \log^\alpha(e+f^p) \,\diff\mu \leq 1.
\end{equation}
The inequality $\log(e+x)\geq 1$ for $x>0$ implies that $\rmint_\Omega f^p\,\diff\mu \leq 1$. To control the term $\mathrm{Ent}^\alpha(f^p)$, observe that if $A=\big\{ f^p > \rmint_\Omega f^p\,\diff\mu\big\}$, we have
\begin{equation}
\mr{Ent}^\alpha(f^p) \leq \rmint_{A} f^p \Big( \log f^p - \log\rmint_\Omega f^p\,\diff\mu\Big)^\alpha \,\diff\mu.
\end{equation}
We will break this integral in two parts. First, notice that
\begin{equation}
 \rmint_{A\cap\{f\leq 1\}} f^p \Big( \log f^p - \log\rmint_\Omega f^p\,\diff\mu\Big)^\alpha \,\diff\mu \leq \Big(\rmint_\Omega f^p\,\diff\mu\Big) \Big(-\log \rmint_\Omega f^p\,\diff\mu\Big)^\alpha \leq \Big(\frac{\alpha}{e}\Big)^\alpha,
\end{equation}
as $x\log^\alpha (1/x)\leq \big(\alpha/e\big)^\alpha$ for $x\in(0,1)$. For the remaining integral we write
\begin{equation}
\begin{split}
 \rmint_{\{f> 1\}}& f^p \Big( \log f^p  - \log\rmint_\Omega f^p\,\diff\mu\Big)^\alpha \,\diff\mu
 \\ & \leq 2^{(\alpha-1)_+} \rmint_\Omega f^p \Big( \log^\alpha f^p - \log^\alpha\Big(\rmint_\Omega f^p\,\diff\mu\Big)\Big) \,\diff \mu \leq 2^{(\alpha-1)_+}\big(1+(\alpha/e)^\alpha\big),
 \end{split}
\end{equation}
where $\gamma_+ = \max\{\gamma,0\}$ and we used the inequality $(A+B)^{\alpha} \leq 2^{(\alpha-1)_+} (A^\alpha+B^\alpha)$ which holds for every $A,B\geq0$. This proves the leftmost inequality of \eqref{eq:orlicz-ineq}.

For the inequality on the right, assume that
\begin{equation}
\max\big\{ \|f\|_{L_p}^p, \mathrm{Ent}^\alpha\big(f^p\big) \big\} \leq 1.
\end{equation}
Then,  since $\rmint_\Omega f^p\,\diff\mu \leq 1$, the monotonicity of $x\mapsto \log^\alpha x$ on $(0,\infty)$ gives
\begin{equation} \label{eq:upper-for-Orlicz}
\rmint_\Omega f^p \log^\alpha f^p \,\diff\mu \leq \rmint_\Omega f^p \log^\alpha\Big( \frac{f^p}{\rmint f^p\,\diff\mu}\Big) \,\diff\mu = \mr{Ent}^\alpha\big(f^p\big) \leq 1.
\end{equation}
At this point we shall use the elementary inequality
\begin{equation} \label{eq:eleme}
y\log^\alpha(e+y) \leq \underbrace{\max\Big\{ \log^\alpha(e+1)+2 (\alpha/e)^\alpha, \tfrac{2^{1/\alpha}+e-1}{2^{1/\alpha}-1}\log^\alpha\Big(e+ \tfrac{2^{1/\alpha}+e-1}{2^{1/\alpha}-1}\Big) \Big\}}_{K_\alpha} +2y\log^\alpha y
\end{equation}
which is valid for all $y>0$. Indeed, assuming \eqref{eq:eleme}, we have
\begin{equation}
\rmint_\Omega f^p \log^\alpha(e+f^p) \, \diff\mu \stackrel{\eqref{eq:eleme}}{\leq} K_\alpha + 2\rmint_\Omega f^p\log^\alpha f^p \,\diff\mu \stackrel{\eqref{eq:upper-for-Orlicz}}{\leq} K_\alpha+2
\end{equation}
and therefore 
\begin{equation}
\rmint_\Omega \frac{f^p}{K_\alpha+2} \log^\alpha\Big(e+\frac{f^p}{K_\alpha+2}\Big) \, \diff\mu \leq \frac{1}{K_\alpha+2} \rmint_\Omega f^p \log^\alpha(e+f^p)\,\diff \mu \leq 1.
\end{equation}
This implies that $\|f\|_{L_p(\log L)^\alpha} \leq (K_\alpha+2)^{1/p}$ which is the desired estimate.

To prove the elementary inequality \eqref{eq:eleme} we consider several cases. If $0<y<1$,
\begin{equation}
y\log^\alpha(e+y) + 2y \log^\alpha(1/y) \leq \log^\alpha(e+1) + 2(\alpha/e)^\alpha
\end{equation}
since the first summand is increasing and the second is bounded by $2(\alpha/e)^\alpha$. One the other hand, if $y\geq \tfrac{2^{1/\alpha}+e-1}{2^{1/\alpha}-1}$ and we denote by $\beta=2^{1/\alpha}>1$, then we have
\begin{equation}
2\log^\alpha y = \log^\alpha y^\beta \geq \log ^\alpha\big( \beta (y-1) + 1\big) > \log^\alpha(y+e)
\end{equation}
and thus \eqref{eq:eleme} holds without the term $K_\alpha$. Finally, if $1<y<\tfrac{2^{1/\alpha}+e-1}{2^{1/\alpha}-1}$, then we just use that
\begin{equation}
y\log^\alpha(e+y) \leq\frac{2^{1/\alpha}+e-1}{2^{1/\alpha}-1}\log^\alpha\Big(e+ \frac{2^{1/\alpha}+e-1}{2^{1/\alpha}-1}\Big)
\end{equation}
and \eqref{eq:eleme} holds without the term $2y\log^\alpha y$. This completes the proof of the lemma.
\end{proof}


\section{Further remarks} \label{sec:suboptimal}

\subsection{Limitations of semigroup methods} \label{sec:sgp}
We include here, mainly for pedagogical reasons,  an argument inspired by \cite{IVV20,CE22} which leads to versions of Theorem \ref{thm:main} with a suboptimal Orlicz space on the left-hand side.  This material was obtained jointly with Ramon van Handel, whom we thank for allowing us to include it here.  We shall need the following technical result which is due to \cite[Lemma 44 and Remark 48]{EI20}. Observe that the case $p=2$ can be proven easily by an expansion in the Walsh basis and Parseval's identity.

\begin{lemma} \label{lem:ei}
For every $p\in(1,\infty]$, there exists $A_p\in(0,\infty)$ such that for any $n\in\N$, every function $h:\ms{C}_n\to\C$ satisfies
\begin{equation} \label{eq:ei}
\forall \ t>0,\qquad \big\| \nabla P_t h\big\|_{L_p} \leq \frac{A_p}{\big(e^{p^*t}-1\big)^{1/p^*}} \|h\|_{L_p},
\end{equation}
where $p^\ast = \min\{p,2\}$.
\end{lemma}

By Talagrand's logarithmic Sobolev inequality \cite{Tal93} and interpolation, we deduce the following Orlicz space version of hypercontractivity.

\begin{lemma}
For every $p\in(1,\infty)$, there exists $\msf{L}_p\in(0,\infty)$ such that the following holds. For any normed space $(E,\|\cdot\|_E)$,  any $\alpha\in[0,p/2]$ and any $n\in\N$, every function $f:\ms{C}_n\to E$ satisfies
\begin{equation} \label{eq:orl-hyp}
\forall \ t\in(0,1), \qquad \big\| P_t f\big\|_{L_p(\log L)^\alpha(E)} \leq \frac{\msf{L}_p}{t^{2\alpha/pp^\ast}} \|f\|_{L_p(E)},
\end{equation}
where $p^\ast=\min\{p,2\}$.
\end{lemma}

\begin{proof}
The desired inequality is obvious for $\alpha=0$ since $P_t$ is a contraction.  Therefore, the general claim would follow from the endpoint case $\alpha=p/2$ by complex interpolation of $L_p(\log L)^\alpha$ Orlicz spaces (which appears explicitly, e.g.,  in \cite{Mey82}).  As $P_t$  has a positive kernel, Jensen's inequality implies the pointwise bound $\|P_t f\|_E \leq P_th$, where $h=\|f\|_E$ and thus we have
\begin{equation}
\big\| P_t f\big\|_{L_p(\log L)^{p/2}(E)} \leq \big\| P_t h\big\|_{L_p(\log L)^{p/2}} \leq \big\| P_t h - \mb{E}h\big\|_{L_p(\log L)^{p/2}} + \mb{E}h.
\end{equation}
As $\mb{E}P_th = \mb{E}h$,  Talagrand's inequality \eqref{eq:talagrand} gives
\begin{equation}
\big\| P_t h - \mb{E}h\big\|_{L_p(\log L)^{p/2}} \stackrel{\eqref{eq:talagrand}}{\leq} \msf{C}_p \big\| \nabla P_t h\big\|_{L_p}  \stackrel{\eqref{eq:ei}}{\leq} \frac{\msf{C}_pA_p}{\big(e^{p^*t}-1\big)^{1/p^*}} \|f\|_{L_p(E)}.
\end{equation}
The proof follows by combining the above since additionally $\mb{E}h = \|f\|_{L_1(E)} \leq \|f\|_{L_p(E)}$.
\end{proof}

Equipped with these lemmas, we can now try to prove Theorem \ref{thm:main} by the methods of \cite{IVV20,CE22}.

\begin{proof} [Attempted proof of Theorem \ref{thm:main}]
Consider a function $f:\ms{C}_n\to E$ and without loss of generality assume that $\mb{E}f=0$.  Then, 
\begin{equation}
\begin{split}
\|f\|_{L_p(\log L)^\alpha(E)} = 2 \Big\| \rmint_0^\infty P_t \Delta P_t f \, \diff t & \Big\|_{L_p(\log L)^\alpha(E)} \leq 2\rmint_0^\infty \big\|P_t\Delta P_t f\big\|_{L_p(\log L)^\alpha(E)} \,\diff t
\\ & \stackrel{\eqref{eq:orl-hyp}}{\leq} \widetilde{\msf{L}}_p \rmint_0^\infty \big\|\Delta P_t f\big\|_{L_p(E)} \max\big\{ t^{-2\alpha/pp^\ast},1\big\} \,\diff t.
\end{split}
\end{equation}
The last estimate follows by combining \eqref{eq:orl-hyp} for $t\in(0,1)$ and the usual hypercontractive inequality 
\begin{equation}
\forall \ t\geq1,\qquad \big\|P_t\Delta P_t f\big\|_{L_p(\log L)^\alpha(E)}  \leq \msf{K}_{p,\alpha} \big\|P_t\Delta P_t f\big\|_{L_{1+e^{2t}(p-1)}(E)} \leq \msf{K}_{p,\alpha} \big\|\Delta P_t f\big\|_{L_{p}(E)},
\end{equation}  
where in the first inequality we used a bound between Orlicz norms which is valid for every function.  By the main technical bound of \cite{IVV20} as expressed in \cite[Proposition~18]{CE22},  we deduce that
\begin{equation}
\|f\|_{L_p(\log L)^\alpha(E)} \leq  \widetilde{\msf{L}}_p \rmint_0^\infty \left(\mb{E} \Big\| \sum_{i=1}^n \delta_i(t) \partial_if\Big\|_{L_p(E)}^p\right)^{1/p} \max\big\{ t^{-2\alpha/pp^\ast},1\big\} \,\frac{\diff t}{\sqrt{e^{2t}-1}},
\end{equation}
where the random variables $\delta_i(t)$ are defined in Remark \ref{rem:general}.  Assuming that the space has cotype $q\in[2,\infty)$, it was shown in \cite[Section~4]{IVV20} that
\begin{equation}
\left(\mb{E} \Big\| \sum_{i=1}^n \delta_i(t) \partial_if\Big\|_{L_p(E)}^p \right)^{1/p} \leq \msf{C}_{p,q} t^{\frac{1}{\max\{p,q\}}-\frac{1}{2}} \left( \mb{E} \Big\| \sum_{i=1}^n \delta_i \partial_if\Big\|_{L_p(E)}^p\right)^{1/p}
\end{equation}
and combining the above we finally get the estimate
\begin{equation*}
\|f\|_{L_p(\log L)^\alpha(E)} \leq  \widetilde{\msf{L}}_p \msf{C}_{p,q}\left( \rmint_0^\infty t^{\frac{1}{\max\{p,q\}}-\frac{1}{2}} \max\big\{ t^{-2\alpha/pp^\ast},1\big\} \,\frac{\diff t}{\sqrt{e^{2t}-1}}\right) \left(\mb{E} \Big\| \sum_{i=1}^n \delta_i \partial_if\Big\|_{L_p(E)}^p\right)^{1/p}.
\end{equation*}
Observe however that this integral converges if and only if $\alpha < \frac{pp^\ast}{2\max\{p,q\}} \leq \frac{p}{2}$.  Therefore, this approach cannot recover the optimal Orlicz space attained in Theorem \ref{thm:main}.
\end{proof}


\subsection{Reductions for vector-valued influence inequalities} Conversely, the arguments used in this paper may shed light to some of the questions left open in \cite{CE22}.  As a matter of fact,  they readily lead to two potentially useful reductions.

\smallskip

\noindent {$\bullet$} In \cite[Theorem~6]{CE22} it was shown that when $E$ is a Banach space of Rademacher type $2$, then for any $\e\in(0,1)$, we have the inequality
\begin{equation} \label{eq:inf}
\big\| f- \mb{E}f\big\|_{L_2(E)}^2 \leq \frac{\msf{C}(E)}{\e} \sum_{i=1}^n \big\|\partial_i f\big\|_{L_2(\log L)^{-1+\e}(E)}^2
\end{equation}
and it remains unknown whether one is allowed to take $\e=0$ as in the scalar case treated in \cite{Tal94} without additional assumptions on the space $E$. Assuming as usual that $\mb{E}f=0$ and denoting $h=\|f\|_E$, we can first use the triangle inequality
\begin{equation}
\begin{split}
\|f\|_{L_2(E)}^2 = \|h\|_{L_2}^2 \leq 2\big( \|h-\mb{E}h\|_{L_2}^2 + |\mb{E}h|^2\big).
\end{split}
\end{equation}
Then,  we observe that by Talagrand's scalar influence inequality the first summand satisfies
\begin{equation}
 \|h-\mb{E}h\|_{L_2}^2 \leq \msf{C} \sum_{i=1}^n \big\|\partial_i h\big\|_{L_2(\log L)^{-1}}^2 \leq \msf{C} \sum_{i=1}^n \big\|\partial_i f\big\|_{L_2(\log L)^{-1}(E)}^2,
\end{equation}
where in the last inequality we used the pointwise bound $|\partial_i h| \leq \|\partial_i f\|_E$. Therefore, in order to prove \eqref{eq:inf} with $\e=0$ it is both necessary and sufficient to prove the strictly weaker estimate
\begin{equation}
\big\| f- \mb{E}f\big\|_{L_1(E)}^2 \leq \msf{C}(E) \sum_{i=1}^n \big\|\partial_i f\big\|_{L_2(\log L)^{-1}(E)}^2
\end{equation}
since the left-hand side is just $|\mb{E}h|^2$.
\smallskip

\noindent {$\bullet$} In \cite{CE22} we also investigated the so-called $L_1-L_p$ inequalities which are morally dual estimates to Talagrand's inequality \eqref{eq:talagrand}.  In the case of scalar-valued functions, \cite[Theorem~33]{CE22} asserts that for any $p\in(1,\infty)$ there exists $\msf{C}_p>0$ such that every $f:\ms{C}_n\to\C$ satisfies
\begin{equation} \label{eq:l1-lp}
\big\|f-\mb{E}f\big\|_{L_p} \leq \msf{C}_p \big\|\nabla f\big\|_{L_p(\log L)^{-p/2}}.
\end{equation}
Moreover,  a semigroup argument similar to that of Section \ref{sec:sgp} proved an analogue of \eqref{eq:l1-lp} with a suboptimal Orlicz space on the right-hand side for functions with values in normed spaces of finite cotype, see \cite[Theorem~29]{CE22}.  An inspection of the proof of Theorem \ref{thm:main} reveals that if one could show a strengthening of \eqref{eq:l1-lp} for real-valued functions with the asymmetric gradient $\msf{M}f$ on the right-hand side, this would also sharpen the vector-valued results of \cite{CE22}.


\subsection{Beckner inequalities}

In \cite{Bec89}, Beckner put forth a family of functional inequalities in Gauss space interpolating between the Poincar\'e and logarithmic Sobolev inequalities. More specifically, he proved that for any $n\in\N$ and $q\in[1,2]$, every function $f:\R^n\to\R$ satisfies
\begin{equation}
\|f\|_{L_2}^2-\|f\|_{L_q}^2 \leq (2-q)\big\|\nabla f\big\|_{L_2}^2.
\end{equation}
Observe that dividing both sides by $2-q$ and taking a limit as $q\to2^-$, one recovers Gross' logarithmic Sobolev inequality \cite{Gro75}. In fact, Beckner-type inequalities are equivalent (up to the value of the constants) to logarithmic Sobolev inequalities. Indeed, the convexity of $r\mapsto \log \|f\|_{L_{1/r}}$ for $r\in(0,1]$ implies that the quantity
\begin{equation}
\frac{1}{\frac{1}{q}-\frac{1}{2}}\Big( \|f\|_{L_2}^2-\|f\|_{L_q}^2 \Big)
\end{equation}
is increasing in $q\in[1,2)$. Combining this observation with Theorem \ref{thm:main}, we deduce the following.

\begin{corollary} \label{cor:beckner}
Let $(E,\|\cdot\|_E)$ be a Banach space of finite cotype.  There exists a constant $\msf{B}(E)\in(0,\infty)$ such that for any $n\in\N$ and $q\in[1,2]$, every function $f:\ms{C}_n\to E$ with $\mb{E}f=0$ satisfies
\begin{equation} \label{eq:beckner}
\big\|f\big\|_{L_2(E)}^2 - \big\|f\big\|_{L_q(E)}^2 \leq \msf{B}(E)(2-q) \rmint_{\ms{C}_n} \Big\| \sum_{i=1}^n \delta_i \partial_i f\Big\|_{L_2(E)}^2 \,\diff\sigma_n(\delta).
\end{equation}
\end{corollary}

One can obtain similar results in Gauss space or the symmetric group, using Theorem \ref{thm:symmetric}.


\subsection{Functional isoperimetric-type inequalities}

Talagrand's isoperimetric inequality from \cite{Tal93} asserts that any Boolean function $f:\{-1,1\}^n\to\{-1,1\}$ satisfies
\begin{equation} \label{eq:isop}
\|\nabla f\|_{L_2} \geq \msf{c} \mr{Var}(f) \sqrt{\log\big(1/\mr{Var}(f)\big)},
\end{equation}
where $\msf{c}>0$ is a universal constant. In the recent work \cite{BIM23}, Beltran, Ivanisvili and Madrid put forth some extensions of \eqref{eq:isop} for functions with values in spaces of finite cotype. Their proofs build on the method of \cite{IVV20} and, similarly to the results of Section \ref{sec:sgp}, the cotype of the target space appears as an exponent in some of their estimates. We present here a result in this spirit which is a direct consequence of Theorem \ref{thm:main} (observe also the analogy with the influence inequality \cite{Tal94,CE22}).

\begin{proposition}
Let $(E,\|\cdot\|_E)$ be a Banach space of finite cotype and $p\in[1,\infty)$.  There exists a constant $\msf{\eta}_p(E)\in(0,\infty)$ such that for any $n\in\N$, every function $f:\ms{C}_n\to E$ with $\mb{E}f=0$ satisfies
\begin{equation} \label{eq:isop2}
\left( \rmint_{\ms{C}_n} \Big\| \sum_{i=1}^n \delta_i \partial_i f\Big\|_{L_p(E)}^p \,\diff\sigma_n(\delta)\right)^{1/p} \geq \eta_p(E)\cdot \|f\|_{L_p(E)} \Big(1+\sqrt{\log\big( \|f\|_{L_p(E)} /\|f\|_{L_1(E)}  \big)}\Big).
\end{equation}
\end{proposition}

\begin{proof}
In view of Theorem \ref{thm:main} it suffices to show that $\|f\|_{L_p(\log L)^{p/2}(E)}$ is bounded below by the right-hand side of \eqref{eq:isop2}. The more general estimate
\begin{equation}
\|f\|_{L_p(\log L)^{\alpha}(E)} \geq \msf{c}_{p,\alpha} \|f\|_{L_p(E)}\Big( 1+\log^{\alpha/p}\big(\|f\|_{L_p(E)}/\|f\|_{L_1(E)}\big) \Big)
\end{equation} 
follows by concatenating \cite[Lemma~17]{CE22} with the H\"older inequality
\begin{equation} \label{eq:horlicz}
\|FG\|_{L_{p/2}} \leq \msf{C}_{\alpha} \|F\|_{L_p(\log L)^\alpha} \|G\|_{L_p(\log L)^{-\alpha}}
\end{equation}
for Orlicz spaces. As we were unable to locate a reference for \eqref{eq:horlicz}, we mention that it follows immediately from the two-point inequality
\begin{equation}
(xy)^{p/2} \leq \tfrac{\msf{C}_\alpha}{2} \big( x^p\log^\alpha(e+x^p) + y^p \log^{-\alpha}(e+y^p)\big),
\end{equation}
which is valid for every $x,y\geq0$ and some fixed constant $\msf{C}_\alpha>0$.
\end{proof}


\bibliographystyle{plain}
\bibliography{LSI-Banach}

\end{document}